\documentclass[12pt,reqno]{article}

\usepackage[usenames]{color}
\usepackage{graphicx}

\usepackage[colorlinks=true,
linkcolor=webgreen, filecolor=webbrown,
citecolor=webgreen]{hyperref}

\definecolor{webgreen}{rgb}{0,.5,0}
\definecolor{webbrown}{rgb}{.6,0,0}

\usepackage{graphics,amsmath,amssymb}
\usepackage{amsthm}
\usepackage{epsf}

\begin{document}

\begin{center}
\epsfxsize=4in

\end{center}

\theoremstyle{plain}
\newtheorem{theorem}{Theorem}
\newtheorem{corollary}[theorem]{Corollary}
\newtheorem{lemma}[theorem]{Lemma}
\newtheorem{proposition}[theorem]{Proposition}

\theoremstyle{definition}
\newtheorem{definition}[theorem]{Definition}
\newtheorem{example}[theorem]{Example}
\newtheorem{conjecture}[theorem]{Conjecture}

\theoremstyle{remark}
\newtheorem{remark}[theorem]{Remark}

\newtheorem {Congruence}{Congruence}

\begin{center}
\vskip 1cm{\LARGE\bf M\'enage Numbers and M\'enage Permutations }
\vskip 1cm \large
Yiting Li\\
Department of Mathematics \\
Brandeis University \\
{\tt yitingli@brandeis.edu}
\end{center}

\newcommand{\Addresses}{{
  \bigskip
  \footnotesize

  Yiting Li, \textsc{Department of Mathematics, Brandeis University,
415 Soutrh Street, Waltham, MA 02453, USA}\par\nopagebreak
  \textit{E-mail address}, Yiting Li: \texttt{yitingli@brandeis.edu}

}}

\vskip .2 in

\newcommand {\stirlingf}[2]{\genfrac[]{0pt}{}{#1}{#2}}
\newcommand {\stirlings}[2]{\genfrac\{\}{0pt}{}{#1}{#2}}

\begin{abstract}
In this paper, we study the combinatorial structures of straight and
ordinary m\'enage permutations. Based on these structures, we prove
four formulas. The first two formulas define a relationship between
the m\'enage numbers and the Catalan numbers. The other two formulas
count the m\'enage permutations by number of cycles.
\end{abstract}

\section{Introduction}
\subsection{Straight m\'enage permutations and straight m\'enage numbers}
The straight m\'enage problem asks for the number of ways one can
arrange $n$ male-female pairs along a linearly arranged table in
such a way that men and women alternate but no woman sits next to
her partner.

We call a permutation $\pi\in S_n$ a $straight$ {\it m\'enage}
$permutation$ if $\pi(i)\ne i$ and $\pi(i)\ne i+1$ for $1\le i\le
n$. Use $V_n$ to denote the number of straight m\'enage permutations
in $S_n$. We call $V_n$ the $nth$ $straight$ {\it menage} $number$.

The straight m\'enage problem is equivalent to finding $V_n$. Label
the seats along the table as $1,2,\ldots,2n$. Sit the men at
positions with even numbers and women at positions with odd numbers.
Let $\pi$ be the permutation such that the man at position $2i$ is
the partner of the woman at position $2\pi(i)-1$ for $1\le i\le n$.
Then, the requirement of the straight m\'enage problem is equivalent
to the condition that $\pi(i)$ is neither $i$ nor $i+1$ for $1\le
i\le n$.

\subsection{Ordinary m\'enage permutations and ordinary m\'enage numbers}
The ordinary m\'enage problem asks for the number of ways one can
arrange $n$ male-female pairs around a circular table in such a way
that men and women alternate, but no woman sits next to her partner.

We call a permutation $\pi\in S_n$ an $ordinary$ {\it m\'enage}
$permutation$ if $\pi(i)\ne i$ and $\pi(i)\not\equiv i+1$ (mod $n$)
for $1\le i\le n$. Use $U_n$ to denote the number of ordinary
m\'enage permutations in $S_n$. We call $U_n$ the $nth$ $ordinary$
{\it m\'enage} $number$.

The ordinary m\'enage problem is equivalent to finding $U_n$. Label
the seats around the table as $1,2,\ldots,2n$. Sit the men at
positions with even numbers and women at positions with odd numbers.
Let $\pi$ be the permutation such that the man at position $2i$ is
the partner of the woman at position $2\pi(i)-1$ for all $1\le i\le
n$. Then, the requirement of the ordinary m\'enage problem is
equivalent to the condition that $\pi(i)$ is neither $i$ nor $i+1$
(mod $n$) for $1\le i\le n$.

We hold the convention that the empty permutation $\pi_\emptyset\in
S_0$ is both a straight m\'enage permutation and an ordinary
m\'enage permutation, so $U_0=V_0=1$.

\subsection{Background}
Lucas \cite{Lucas} first posed the problem of finding ordinary
m\'enage numbers. Touchard \cite{Touchard} first found the following
explicit formula \eqref{eq:known_formulas_for_U_n}. Kaplansky and
Riordan \cite{Kaplansky} also proved an explicit formula for
ordinary m\'enage numbers. For other early work in m\'enage numbers,
see \cite{Kaplansky1943,MW} and references therein. Among more
recent papers, there are some using bijective methods to study
m\'enage numbers. For example, Canfield and Wormald \cite{CW} used
graphs to address the question. One can find the following formulas
of m\'enage numbers in \cite{Bogart,Touchard} and in Chapter 8 of
\cite{Riordan}:
\begin{align}\label{eq:known_formulas_for_U_n}
U_m=\sum\limits_{k=0}^m(-1)^k\dfrac{2m}{2m-k}{2m-k\choose
k}(m-k)!\quad\quad(m\ge2);\\
V_n=\sum\limits_{k=0}^n(-1)^k{2n-k\choose
k}(n-k)!\quad\quad(n\ge0).\label{eq:known_formulas_for_V_n}
\end{align}

The purpose of the current paper is to study the combinatorial
structures of straight and ordinary m\'enage permutations and to use
these structures to prove some formulas of straight and ordinary
m\'enage numbers. We also give an analytical proof of Theorem
\ref{thm:main_theorem_1} in Section \ref{appendix}.

\subsection{Main results}

Let $C_k$ be the $k$th $Catalan$ $number$:
\[
C_k=\dfrac{(2k)!}{k!\,(k+1)!}
\]
and $c(x)=\sum\limits_{k=0}^\infty c_kx^k$. Our first main result is
the following theorem.
\begin{theorem}\label{thm:main_theorem_1}
\begin{align}\label{eq:result_of_straight_menage_numbers}
\sum\limits_{n=0}^{\infty}n!\,x^n=\sum\limits_{n=0}^\infty
V_nx^nc(x)^{2n+1}, \\
\sum\limits_{n=0}^{\infty}n!\,x^n=c(x)+c'(x)\sum\limits_{n=1}^\infty
U_nx^nc(x)^{2n-2}.\label{eq:result_of_ordinary_menage_numbers}
\end{align}
\end{theorem}

Our second main result counts the straight and ordinary m\'enage
permutations by the number of cycles.

For $k\in\mathbb{N}$, use $(\alpha)_k$ to denote
$\alpha(\alpha+1)\cdots(\alpha+k-1)$. Define $(\alpha)_0=1$. For
$k\le n$, use $C_n^k$ ($D_n^k$) to denote the number of straight
(ordinary) m\'enage permutations in $S_n$ with $k$ cycles.
\begin{theorem}
\begin{align}\label{eq:straight_by_cycles}
1+\sum_{n=1}^\infty\sum_{j=1}^n
C_n^j\alpha^jx^n=\sum\limits_{n=0}^\infty(\alpha)_n\frac{x^n}{(1+x)^n(1+\alpha
x)^{n+1}},
\end{align}
\begin{align}\label{eq:ordinary_by_cycles} 1+\sum_{n=1}^\infty\sum_{j=1}^n
D_n^j\alpha^jx^n=\frac{x+\alpha x^2}{1+x}+(1-\alpha
x^2)\sum\limits_{n=0}^\infty(\alpha)_n\frac{x^n}{(1+x)^{n+1}(1+\alpha
x)^{n+1}}.
\end{align}
\end{theorem}

\subsection{Outline}
We give some preliminary concepts and facts in Section
\ref{sec:preliminaries}. In Section
\ref{sec:reductions_and_nice_bijections}, we define three types of
reductions and the nice bijection. Then, we study the structure of
straight m\'enage permutations and prove
\eqref{eq:result_of_straight_menage_numbers} in Section
\ref{sec:straight_menage_permutation}. In Section
\ref{sec:ordinary_menage_permutation}, we study the structure of
ordinary m\'enage permutations and prove
\eqref{eq:result_of_ordinary_menage_numbers}. Finally, we count the
straight and ordinary m\'enage permutations by number of cycles and
prove \eqref{eq:straight_by_cycles} and
\eqref{eq:ordinary_by_cycles} in Section
\ref{count_permutations_by_cycles}.

\section{Preliminaries}\label{sec:preliminaries}

For $n\in\mathbb{N}$, we use $[n]$ to denote $\{1,\ldots,n\}$.
Define $[0]$ to be $\emptyset$.

\begin{definition}
Suppose $n>0$ and $\pi\in S_n$. If $\pi(i)=i+1$, then we call
$\{i,i+1\}$ a $succession$ of $\pi$. If $\pi(i)\equiv i+1$ (mod
$n$), then we call $\{i,\pi(i)\}$ a $generalized$ $succession$ of
$\pi$.
\end{definition}
\subsection{Partitions and Catalan numbers}\label{sec:partitions_and_Catalan_numbers}

Suppose $n>0$. A $partition$ $\epsilon$ of $[n]$ is a collection of
disjoint subsets of $[n]$ whose union is $[n]$. We call each subset
a $block$ of $\epsilon$. We also describe a partition as an
equivalence relation: $p\sim_\epsilon q$ if and only if $p$ and $q$
belong to a same block of $\epsilon$.

If a partition $\epsilon$ satisfies that for any $p\sim_\epsilon p'$
and $q\sim_\epsilon q'$, $p<q<p'<q'$ implies $p\sim_\epsilon q$;
then, we call $\epsilon$ a $noncrossing$ $partition$.

For $n\in\mathbb{N}$, suppose $\epsilon=\{V_1,\ldots,V_k\}$ is a
noncrossing partition of $[n]$ and $V_i=\{a_1^i,\ldots,
a_{j_i}^i\}$, where $a_1^i<\cdots<a_{j_i}^i$. Then, $\epsilon$
induces a permutation $\pi\in S_n$: $\pi(a_{r(i)}^i)=a_{r(i)+1}^i$
for $1\le r(i)\le j_i-1$ and $\pi(a_{j_i}^i)=a_1^i$. It is not
difficult to see that different noncrossing partitions induce
different permutations.

The following lemma is well known. See, for example, \cite{Stanley}.
\begin{lemma}
For $n\in\mathbb{N}$, there are $C_n$ noncrossing partitions of
$[n]$.
\end{lemma}
It is well known that the generating function of the Catalan numbers
is
$$c(x)=\sum\limits_{n=0}^\infty C_nx^n=\dfrac{1-\sqrt{1-4x}}{2x}.$$
It is also well known that one can define the Catalan numbers by
recurrence relation
\begin{align}\label{recurrence_of_Catalan_number}
C_{n+1}=\sum\limits_{k=0}^nC_kC_{n-k}
\end{align}
with initial condition $C_0=1$.
\begin{lemma}\label{thm:property_of_Catalan_numbers}
The generating function of the Catalan numbers $c(x)$ satisfies
$$c(x)=\frac{1}{1-xc(x)}=1+xc^2(x)\quad\text{and}\quad\dfrac{c^3(x)}{1-xc^2(x)}=c'(x).$$
\end{lemma}

\begin{proof}[Proof of Lemma \ref{thm:property_of_Catalan_numbers}]
The first formula is well known. By the first formula,
\begin{align}\label{www}
c'(x)=c^2(x)+2xc(x)c'(x)=\dfrac{c^2(x)}{1-2xc(x)}.
\end{align}
Thus, to prove the second formula, we only have to show that
$\dfrac{c(x)}{1-xc^2(x)}=\dfrac{1}{1-2xc(x)}$ which is equivalent to
$c(x)-2xc^2(x)=1-xc^2(x)$. This follows from the first formula.
\end{proof}

\subsection{Diagram representation of permutations}\label{sec:diagrams}
\subsubsection{Diagram of horizontal type}
For $n>0$ and $\pi\in S_n$, we use a diagram of horizontal type to
represent $\pi$. To do this, draw $n$ points on a horizontal line.
The points represent the numbers $1,\ldots,n$ from left to right.
For each $i\in[n]$, we draw a directed arc from $i$ to $\pi(i)$. The
permutation uniquely determines the diagram. For example, if
$\pi=(1,5,4)(2)(3)(6)$, then its diagram is

\centerline{\includegraphics[width=3.5in]{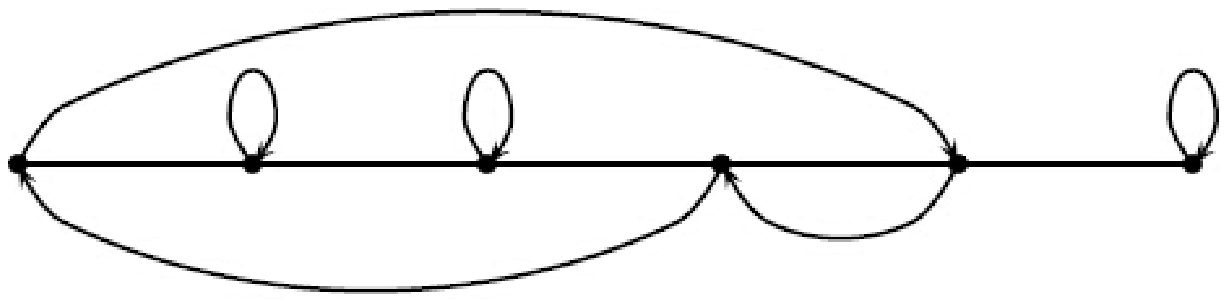}}

\subsubsection{Diagram of circular type}
For $n>0$ and $\pi\in S_n$, we also use a diagram of circular type
to represent $\pi$. To do this, draw $n$ points uniformly
distributed on a circle. Specify a point that represents the number
1. The other points represent $2,\ldots,n$ in counter-clockwise
order. For each $i$, draw a directed arc from $i$ to $\pi(i)$. The
permutation uniquely determines the diagram (up to rotation). For
example, if $\pi=(1,5,4)(2)(3)(6)$, then its diagram is

\centerline{\includegraphics[width=1.5in]{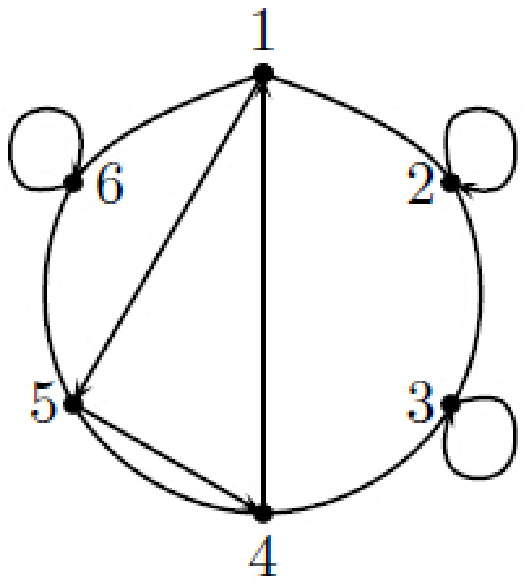}}

\subsection{The empty permutation}

The empty permutation $\pi_\emptyset\in S_0$ is a permutation with
no fixed points, no (generalized) successions and no cycles.

\section{Reductions and nice bijections}\label{sec:reductions_and_nice_bijections} In this section, we introduce reductions and nice
bijections which serve as our main tools to study m\'enage
permutations.
\subsection{Reduction of type 1}
Intuitively speaking, to perform a reduction of type 1 is to remove
a fixed point from a permutation. Suppose $n\ge1$, $\pi\in S_n$ and
$\pi(i)=i$. Define $\pi'\in S_{n-1}$ such that:
\begin{align}\label{eq:expression_of_reduction_of_type_1}
\pi'(j)=\begin{cases}\pi(j)&\text{if}\quad j<i\quad\text{and}\quad\pi(j)<i;\\
\pi(j)-1&\text{if}\quad
j<i\quad\text{and}\quad\pi(j)>i;\\\pi(j+1)&\text{if}\quad j\ge
i\quad\text{and}\quad\pi(j+1)<i;\\\pi(j+1)-1&\text{if}\quad j\ge
i\quad\text{and}\quad\pi(j+1)>i\end{cases}
\end{align}
when $n>1$. When $n=1$, define $\pi'$ to be $\pi_\emptyset$. If we
represent $\pi$ by a diagram (of either type), erase the point
corresponding to $i$ and the arc connected to the point (and number
other points appropriately for the circular case); then we obtain
the diagram of $\pi'$. We call this procedure of obtaining a new
permutation by removing a fixed point a $reduction$ $of$ $type$ $1$.
For example, if
$$\pi=(1,5,6,4)(2)(3)(7)\in S_7,$$ then by removing the fixed point 3
we obtain $\pi'=(1,4,5,3)(2)(6)\in S_6$.

\subsection{Reduction of type 2}
Intuitively speaking, to do a reduction of type 2 is to glue a
succession $\{k,k+1\}$ together. Suppose $n\ge2$, $\pi\in S_n$ and
$\pi(i)=i+1$. Define $\pi'\in S_{n-1}$ such that:
\begin{align}\label{eq:expression_of_reduction_of_type_2}
\pi'(j)=\begin{cases}\pi(j)&\text{if}\quad j<i\quad\text{and}\quad\pi(j)\le i;\\
\pi(j)-1&\text{if}\quad
j<i\quad\text{and}\quad\pi(j)>i+1;\\\pi(j+1)&\text{if}\quad j\ge
i\quad\text{and}\quad\pi(j+1)\le i;\\\pi(j+1)-1&\text{if}\quad j\ge
i\quad\text{and}\quad\pi(j+1)>i+1.\end{cases}
\end{align}
If we represent $\pi$ by the diagram of the $horizontal$ type, erase
the arc from $i$ to $i+1$, and glue the points corresponding to $i$
and $i+1$ together; then, we obtain the diagram of $\pi'$. We call
this procedure of obtaining a new permutation by gluing a succession
together a $reduction$ $of$ $type$ $2$. For example, if
$$\pi=(1,5,6,4)(2)(3)(7)\in S_7,$$ then by gluing 5 and 6 together,
we obtain $\pi'=(1,5,4)(2)(3)(6)\in S_6$.

\subsection{Reduction of type 3}
Intuitively speaking, to perform a reduction of type 3 is to glue a
generalized succession $\{k,k+1\pmod{n}\}$ together. Suppose
$n\ge1$, $\pi\in S_n$ and $\pi(i)\equiv i+1\pmod{n}$. Define $\pi'$
to be the same as in \eqref{eq:expression_of_reduction_of_type_2}
when $i\ne n$. When $i=n>1$, define $\pi'$ to be
\begin{align*}
\pi'(j)=\begin{cases}\pi(j)&\text{ if }j\ne\pi^{-1}(n);\\1&\text{ if
}j=\pi^{-1}(n).\end{cases}
\end{align*}
When $i=n=1$, define $\pi'$ to be $\pi_\emptyset$. If we represent
$\pi$ by a diagram of the $circular$ type, erase the arc
 from $i$ to $i+1$ (mod $n$), glue the points corresponding to $i$
and $i+1$ (mod $n$) together and number the points appropriately;
then, we obtain the diagram of $\pi'$. We call this procedure of
obtaining a new permutation by gluing a generalized succession
together a $reduction$ $of$ $type$ $3$. For example, if
$$\pi=(1,5,6,7)(2)(3)(4)\in S_7,$$ then by gluing 1 and 7 together,
we obtain $\pi'=(1,5,6)(2)(3)(4)\in S_6$.

\subsection{Nice bijections}\label{sec:nice_bijection}

Suppose $n\ge1$ and $f$ is a bijection from $[n]$ to $\{2,\ldots,
n+1\}$. We can also represent $f$ by a diagram of horizontal type as
for permutations. The bijection uniquely determines the diagram. If
$f$ has a fixed point or there exists $i$ such that $f(i)=i+1$, then
we can also perform reductions of type 1 or type 2 on $f$ as above.
In the latter case, we also call $\{i,i+1\}$ a $succession$ of $f$.
We can reduce $f$ to a bijection with no fixed points and no
successions by a series of reductions. It is easy to see that the
resulting bijection does not depend on the order of the reductions.

The following diagram shows an example of reduction of type 2 on the
bijection by gluing the succession 2 and 3 together.

\centerline{\includegraphics[width=4in]{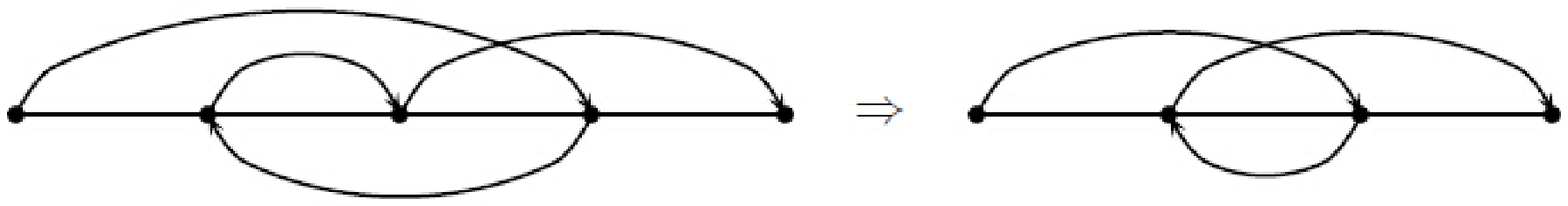}}
\begin{definition}
Suppose $f$ is a bijection from $[n]$ to $\{2,\ldots, n+1\}$. If
there exist a series of reductions of type 1 or type 2 by which one
can reduce $f$ to the simplest bijection $1\mapsto2$, then we call
$f$ a $nice$ $bijection$.
\end{definition}
Suppose $\pi\in S_n$ and $p$ is a point of $\pi$. We can replace $p$
by a bijection $f$ from $[k]$ to $\{2,\ldots, k+1\}$ and obtain a
new permutation $\pi'\in S_{n+k}$ by the following steps:

(1) represent $\pi$ by the horizontal diagram;

(2) add a point $q$ right before $p$ and add an arc from $q$ to $p$
;

(3) replace the arc from $\pi^{-1}(p)$ to $p$ by an arc from
$\pi^{-1}(p)$ to $q$ ;

(4) replace the arc from $q$ to $p$ by the diagram of $f$.

For example, if $\pi$, $p$ and $f$ are as below,

\centerline{\includegraphics[width=4in]{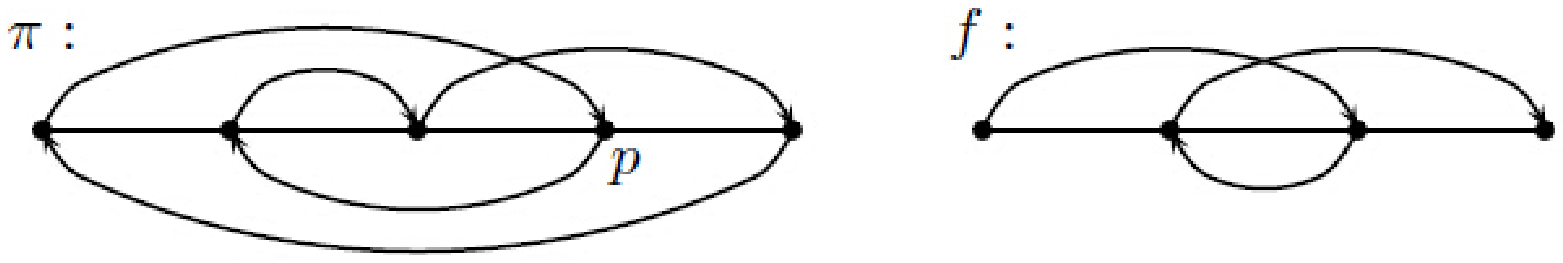}}

then, we can replace $p$ by $f$ and obtain the following
permutation:

\centerline{\includegraphics[width=3.5in]{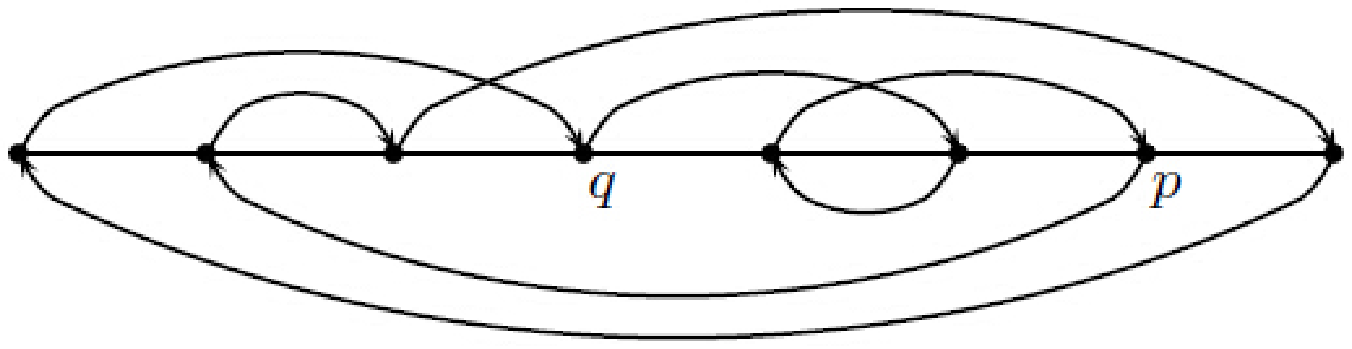}}

It is easy to see, by the definition of the nice bijection, that if
$f$ is a nice bijection, then we can reduce $\pi'$ back to $\pi$:
first, we can reduce $\pi'$ to the permutation obtained in Step (3)
because $f$ is nice; then, by gluing the succession $p$ and $q$
together, we obtain $\pi$. Notice that the bijection $f$ in the
above example is not nice.

For the circular diagram of a permutation $\pi$, we can also use
Steps (2)--(4) shown above to obtain a new diagram. However, to
obtain a permutation $\pi'$, we need to specify the point
representing number 1 in the new diagram. See the example in the
proof of Theorem \ref{gf_of_r}.

\begin{lemma}\label{lemma:count_nice_bijection}
For $n\ge2$, let $a_n$ be the number of nice bijections from $[n-1]$
to $\{2,\ldots, n\}$. Define $a_1=1$. The generating function of
$a_n$ satisfies the following equation:
\[
g(x):=a_1x+a_2x^2+a_3x^3+\cdots=xc(x)
\]
where $c(x)=\sum\limits_{n=0}^\infty C_kx^k$ is the generating
function of the Catalan numbers.
\end{lemma}

\begin{proof}[Proof of Lemma \ref{lemma:count_nice_bijection}]
Suppose $n\ge2$ and $f$ is a nice bijection from $[n-1]$ to
$\{2,\ldots, n\}$ such that $f(1)=k$.

Suppose $p<k$. We claim that $f(p)<k$. If not, suppose $q=f(p)>k$.
Then, neither $\{1,k\}$ nor $\{p,q\}$ is a succession. Consider the
horizontal diagram of $f$. If we perform a reduction of type 1 or
type 2 on $f$, then the arc we remove is neither the arc from $1$ to
$k$ nor the arc from $p$ to $q$. By induction, no matter how many
reductions we perform, there always exists an arc from $1$ to $k'$
and an arc from $p'$ to $q'$ such that $p'<k'<q'$. In other words,
we can never reduce the bijection to $1\mapsto2$. Thus, we proved
the claim.

Thus, the image of $\{2,\ldots, k-1\}$ under $f$ must be
$\{2,\ldots, k-1\}$, and therefore, the image of $\{k,\ldots,n-1\}$
under $f$ must be $\{k+1,\ldots,n\}$.

This implies that $f|_{\{2,\ldots, k-1\}}$ has the same diagram as a
permutation $\tau\in S_{k-2}$ and we can reduce $\tau$ to
$\pi_\emptyset$. This also implies that $f|_{\{k,\ldots, n-1\}}$ has
the same diagram as a nice bijection
 from $[n-k]$ to $\{2,\ldots, n-k+1\}$. By Lemma
\ref{lemma:a_permutation_can_be_reduced_to_null_iff...}, the number
of nice bijections from $[n-1]$ to $\{2,\ldots, n\}$ such that
$f(1)=k$ is $C_{k-2}a_{n-k+1}$. Letting $k$ vary, we obtain
$a_n=\sum\limits_{k=2}^nC_{k-2}a_{n-k+1}$.

By \eqref{recurrence_of_Catalan_number}, $(C_n)_{n\ge0}$ and
$(a_{n+1})_{n\ge0}$ have the same recurrence relation and the same
initial condition $C_0=a_1=1$, so $C_n=a_{n+1}(n\ge0)$. Therefore
$g(x)=xc(x)$.
\end{proof}

\section{Structure of straight m\'enage permutations}\label{sec:straight_menage_permutation} In Section \ref{sec:straight_menage_permutation}, when
we mention a reduction, we mean a reduction of \textbf{type 1 or
type 2}.

If a permutation $\pi$ is not a straight m\'enage permutation, then
$\pi$ has at least one fixed point or succession. Thus, we can apply
a reduction to $\pi$. By induction, we can reduce $\pi$ to a
straight m\'enage permutation $\pi'$ by a series of reductions. For
example, we can reduce $\pi_1=(1,3)(2)(4,5,6)$ to $\pi_\emptyset$:
$(1,3)(2)(4,5,6)\to(1,2)(3,4,5)\to(1)(2,3,4)\to(1,2,3)\to(1,2)\to(1)\to\pi_\emptyset$.
We can reduce $\pi_2=(1,5,4)(2)(3)(6)\in S_6$ to $(1,3,2)$:
$(1,5,4)(2)(3)(6)\to(1,5,4)(2)(3)\to(1,4,3)(2)\to(1,3,2)$. It is
easy to see that the resulting straight m\'enage permutation does
not depend on the order of the reductions. Recall that we defined
the permutation induced from a noncrossing partition in Section
\ref{sec:partitions_and_Catalan_numbers}.
\begin{lemma}\label{lemma:a_permutation_can_be_reduced_to_null_iff...}
Suppose $\pi\in S_n$. We can reduce $\pi$ to $\pi_\emptyset$ by
reductions of type 1 and type 2 if and only if there is a
noncrossing partition inducing $\pi$. In particular, there are $C_n$
such permutations in $S_n$.
\end{lemma}

\begin{proof}
$\Rightarrow$:  Suppose we can reduce $\pi\in S_n$ to
$\pi_\emptyset$. Then, $\pi$ has at least one fixed point or
succession. We use induction on $n$. If $n=1$, the conclusion is
trivial. Suppose $n>1$.

If $\pi$ has a fixed point $i$, then by reduction of type 1 on $i$
we obtain $\pi'$ satisfying
\eqref{eq:expression_of_reduction_of_type_1}. By induction
assumption, there is a noncrossing partition
$\Phi=\{V_1,\ldots,V_k\}$ inducing $\pi'$. Now, we define a new
noncrossing partition $\Pi_1(\Phi,i)$ as follows. Set
$$\tilde V_r=\{x+1|x\in V_r\text{ and }x\ge
i\}\cup\{x|x\in V_r\text{ and }x<i\}$$ for $1\le r\le k$ and $\tilde
V_{k+1}=\{i\}$. Define $\Pi_1(\Phi,i)=\{\tilde V_1,\ldots,\tilde
V_{k+1}\}$. It is not difficult to check that $\Pi_1(\Phi,i)$ is a
noncrossing partition inducing $\pi$.

If $\pi$ has a succession $\{i,i+1\}$, then by reduction of type 2
on $\{i,i+1\}$, we obtain $\pi''$ satisfying
\eqref{eq:expression_of_reduction_of_type_2}. By induction
assumption, there is a noncrossing partition
$\Phi=\{U_1,\ldots,U_s\}$ inducing $\pi''$. Now, we define a new
noncrossing partition $\Pi_2(\Phi,i)$ as follows. Set
\begin{align*}
\tilde U_t=\begin{cases}\{x+1|x\in U_t\text{ and }x> i\}\cup\{x|x\in
U_t\text{ and }x<i\}&\text{ if }t\ne t_0;\\\{x+1|x\in U_t\text{ and
}x> i\}\cup\{x|x\in U_t\text{ and }x<i\}\cup\{i,i+1\}&\text{ if
}t=t_0.\end{cases}
\end{align*}
Define $\Pi_2(\Phi,i)=\{\tilde U_1,\ldots,\tilde U_s\}$. It is not
difficult to check that $\Pi_2(\Phi,i)$ is a noncrossing partition
inducing $\pi$.

$\Leftarrow$: Suppose there is a noncrossing partition
$\{V_1,\ldots,V_k\}$ inducing $\pi$ where $V_r=\{a_1^r,\ldots,
a_{j_r}^r\}$ and $a_1^r<\cdots<a_{j_r}^r$. We prove by using
induction on $n$. The case $n=1$ is trivial. Suppose $n>1$. For
$r_1\ne r_2$, if
$[a_1^{r_1},a_{j_{r_1}}^{r_1}]\cap[a_1^{r_2},a_{j_{r_2}}^{r_2}]\ne\emptyset$,
then either
$[a_1^{r_1},a_{j_{r_1}}^{r_1}]\subset[a_1^{r_2},a_{j_{r_2}}^{r_2}]$
or
$[a_1^{r_2},a_{j_{r_2}}^{r_2}]\subset[a_1^{r_1},a_{j_{r_1}}^{r_1}]$;
otherwise, the partition cannot be noncrossing. Thus, there exists
$p$ such that $[a_1^p,a_{j_p}^p]\cap[a_1^q,a_{j_q}^q]=\emptyset$ for
all $q\ne p$. If $j_p=1$, then $a_1^p$ is a fixed point of $\pi$. If
$j_p>1$, then $\{a_1^p,a_2^p\}$ is a succession of $\pi$.

For the case $j_p=1$, perform a reduction of type 1 on $a_1^p$ and
obtain $\pi'\in S_{n-1}$. Then, $\{\tilde V_r|r\ne p\}$ is a
noncrossing partition inducing $\pi'$, where
\begin{align}\label{eq:tilde_V_r}
\tilde V_r=\{x-1|x\in V_r\text{ and }x> a_{1}^p\}\cup\{x|x\in
V_r\text{ and }x<a_{1}^p\}.
\end{align}
By induction assumption, we can reduce $\pi'$ to $\pi_\emptyset$;
then, we can also reduce $\pi$ to $\pi_\emptyset$.

For the case $j_p>1$, perform a reduction of type 2 on
$\{a_1^p,a_2^p\}$ and get $\pi'\in S_{n-1}$. Then, $\{\tilde
V_r|1\le r\le k\}$ is a noncrossing partition inducing $\pi'$, where
$\tilde V_r$ is the same as in \eqref{eq:tilde_V_r}. By induction
assumption, we can reduce $\pi'$ to $\pi_\emptyset$; then, we can
also reduce $\pi$ to $\pi_\emptyset$.
\end{proof}

Conversely, for a given straight m\'enage permutation $\pi\in S_m$,
what is the cardinality of the set
\begin{align}\label{set:permutations_which_can_be_reduced_to_a_given_permutation}
\{\tau\in S_{m+n}\big|\text{we can reduce }\tau\text{ to }\pi\text{
by reductions of type 1 and type 2}\}?
\end{align}
Interestingly the answer only depends on $m$ and $n$; it does not
depend on the choice of $\pi$. In fact, we have
\begin{theorem}\label{gf_of_omega}
Suppose $m\ge0$ and $\pi\in S_m$ is a straight m\'enage permutation.
Suppose $n\ge0$ and $w_m^n$ is the cardinality of the set in
\eqref{set:permutations_which_can_be_reduced_to_a_given_permutation}.
Set $W_m(x)=w_m^0+w_m^1x+w_m^2x^2+w_m^3x^3+\cdots$; then,
\[
W_m(x)=c(x)^{2m+1}.
\]
\end{theorem}

\begin{proof}[Proof of Theorem \ref{gf_of_omega}]
If $m=0$, then $\pi=\pi_\emptyset$, and the conclusion follows from
Lemma \ref{lemma:a_permutation_can_be_reduced_to_null_iff...}. Thus,
we only consider the case that $m>0$.

Obviously, $w_m^0=1$. Now, suppose $n\ge1$.

Represent $\pi$ by a horizontal diagram. The diagram has $m+1$ gaps:
one gap before the first point, one gap after the last point and one
gap between each pair of adjacent points.

Let $A$ be the set in
\eqref{set:permutations_which_can_be_reduced_to_a_given_permutation}.
To obtain a permutation in $A$, we add points to $\pi$ in the
following two ways:
\begin{enumerate}
\item[(a)] add a permutation induced by a noncrossing partition $\Phi_p$ of
$[d_p]$ into the $p$th gap of $\pi$, where $1\le p\le m+1$ and
$d_p\ge0$ ($d_p=0$ means that we add nothing into the $p$th gap);
\item[(b)] replace the $q$th point of $\pi$ by a nice bijection $f_q$ from
$[r_q]$ to $\{2,\ldots,r_q+1\}$, where $1\le q\le m$ and $r_q\ge0$
($r_q=0$ means that we do not change the $q$th point).
\end{enumerate}
For example, if $\pi$ and $f$ are as below,

\centerline{\includegraphics[width=4in]{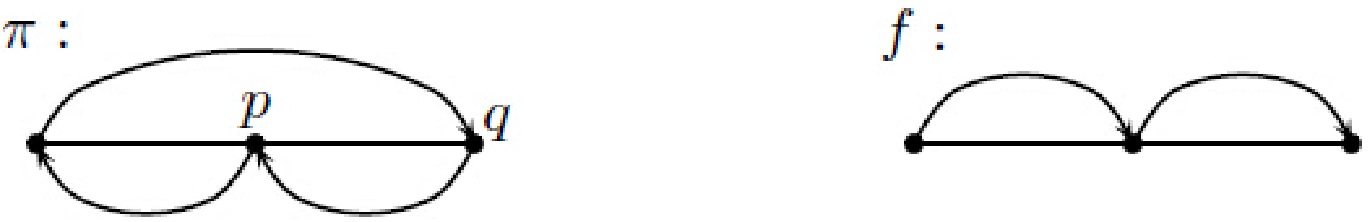}}

then we can add the permutation $(1,2)$ between $p$ and $q$, add the
permutation $(1)(2,3)$ after the last point and replace $p$ by $f$.
Then, we obtain a permutation in $A$ that is:

\centerline{\includegraphics[width=4in]{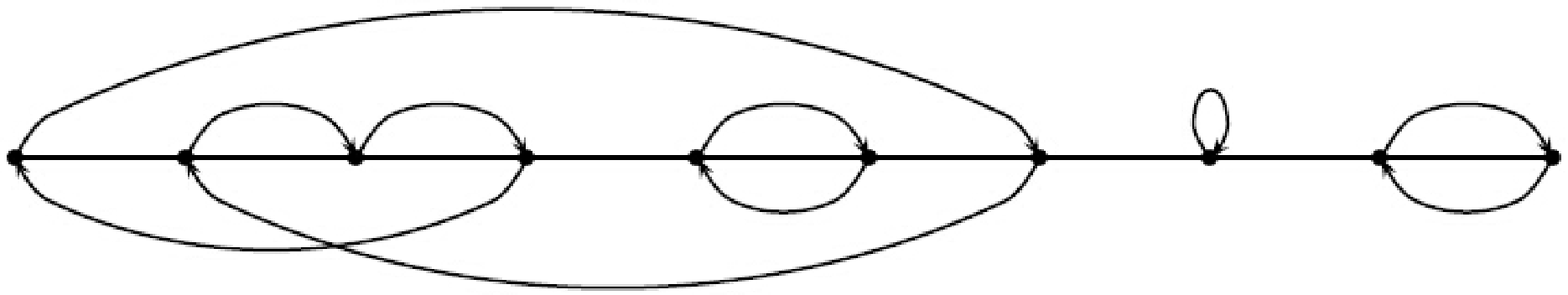}}

\textbf{Statement: The set of permutations constructed from (a) and
(b) equals $A$.} It is easy to see that we can reduce a permutation
constructed through (a) and (b) to $\pi$. Conversely, suppose we can
reduce $\pi'\in S_{m+n}$ to $\pi$. Now, we show that one can
construct $\pi'$ through (a) and (b). Use induction on $n$. The case
that $n=0$ is trivial. Suppose $n>0$. Then, $\pi'$ has at least one
fixed point or succession. For a nice bijection $f$ from $[s]$ to
$\{2,\ldots,s+1\}$ and $1<w_1<s+1$, $1\le w_2\le s+1$, define nice
bijections $B_1(f,w_1)$ and $B_2(f,w_2)$ as
\begin{align}\label{eq:B_1(f,w_1)}
B_1(f,w_1)=\begin{cases}f(x)&\text{ if }x<w_1\text{ and }
f(x)<w_1;\\f(x)+1&\text{ if }x<w_1\text{ and } f(x)\ge
w_1;\\f(x-1)&\text{ if }x>w_1\text{ and } f(x)<w_1;\\f(x-1)+1&\text{
if }x>w_1\text{ and } f(x)\ge w_1;\\w_1&\text{ if }x=w_1\end{cases}
\end{align}
\begin{align}
B_2(f,w_2)=\begin{cases}f(x)&\text{ if }x<w_2\text{ and } f(x)\le
w_2;\\f(x)+1&\text{ if }x<w_2\text{ and } f(x)>w_2;\\f(x-1)&\text{
if }x>w_2\text{ and } f(x)\le w_2;\\f(x-1)+1&\text{ if }x>w_2\text{
and } f(x)>w_2;\\w_2+1&\text{ if
}x=w_2.\end{cases}\label{eq:B_2(f,w_2)}
\end{align}
We can reduce $B_1(f,w_1)$ to $f$ by a reduction of type 1 on the
fixed point $w_1$. We can reduce $B_2(f,w_2)$ to $f$ by a reduction
of type 2 on the succession $\{w_2,w_2+1\}$.

\textbf{Case 1: $\pi'$ has a fixed point $i$.} Using a reduction of
type 1 on $i$, we obtain $\pi''\in S_{m+n-1}$. By induction
assumption, we can construct $\pi''$ from $\pi$ by (a) and (b).
According to the value of $i$, there is either a $k$ such that
\begin{align}\label{eq:condition_of_k}
1\le i-\sum_{j<k}(d_j+r_j+1)\le d_k+1
\end{align}
or a $k'$ such that
\begin{align}\label{eq:condition_of_k'}
1<i-(\sum_{j<k'}(d_j+r_j+1)+d_{k'})\le r_{k'}+1.
\end{align}
If there is a $k$ such that \eqref{eq:condition_of_k} holds, then we
can construct $\pi'$ from $\pi$ by (a) and (b), except that we add
the permutation induced by $\Pi_1(\Phi_k,i-\sum_{j<k}(d_j+r_j+1))$
instead of $\Phi_k$ into the $k$th gap, where $\Pi_1$ is the same as
in the proof of Lemma
\ref{lemma:a_permutation_can_be_reduced_to_null_iff...}. Conversely,
if there exists a $k'$ such that \eqref{eq:condition_of_k'} holds,
then we can construct $\pi'$
 from $\pi$ by (a) and (b), except that we replace the $k'$th
point of $\pi$ by $B_1(f_{k'},i-\sum_{j<k'}(d_j+r_j+1))$ instead of
$f_{k'}$, where $B_1$ is the same as in \eqref{eq:B_1(f,w_1)}.

\textbf{Case 2: $\pi'$ has a succession $\{i,i+1\}$.} By reduction
of type 2 on $\{i,i+1\}$ we obtain $\pi''\in S_{m+n-1}$. By
induction assumption, we can construct $\pi''$ from $\pi$ by (a) and
(b). According to the value of $i$, there is either a $k$ such that
\begin{align}\label{eq:condition_of_k_for_succession}
0<i-\sum_{j<k}(d_j+r_j+1)\le d_k
\end{align}
or a $k'$ such that
\begin{align}\label{eq:condition_of_k'_for_succession}
0<i-(\sum_{j<k'}(d_j+r_j+1)+d_{k'})\le r_{k'}+1.
\end{align}
If there is a $k$ such that \eqref{eq:condition_of_k_for_succession}
holds, then we can construct $\pi'$ from $\pi$ by (a) and (b),
except that we add the permutation induced by
$\Pi_2(\Phi_k,i-\sum_{j<k}(d_j+r_j+1))$ instead of $\Phi_k$ into the
$k$th gap, where $\Pi_2$ is the same as in the proof of Lemma
\ref{lemma:a_permutation_can_be_reduced_to_null_iff...}. Conversely,
if there exists a $k'$ such that
\eqref{eq:condition_of_k'_for_succession} holds, then we can
construct $\pi'$ from $\pi$ by (a) and (b), except that we replace
the $k'$th point of $\pi$ by $B_2(f_{k'},i-\sum_{j<k'}(d_j+r_j+1))$
instead of $f_{k'}$, where $B_2$ is the same as in
\eqref{eq:B_2(f,w_2)}.

Thus, we have proved the statement.

Now, add points to $\pi$ by (a) and (b). The total number of points
added to $\pi$ is $d_1+\cdots+d_{m+1}+r_1+\cdots+r_m$. To obtain a
permutation in $S_{m+n}\cap A$, $d_1+\cdots+d_{m+1}+r_1+\cdots+r_m$
should be $n$. Therefore, the number of permutations in $S_{m+n}\cap
A$ is
\[
w_m^n=\sum\limits_{d_1,\ldots,d_{m+1}\atop r_1,\ldots,
r_m}C_{d_1}\cdots C_{d_{m+1}}a_{1+r_1}\cdots a_{1+r_m}
\]
where $C_k$ is the $k$th Catalan number and $a_k$ is the same as in
Lemma \ref{lemma:count_nice_bijection} and the sum runs over all
$(2m+1)$-triples $(d_1,\ldots,d_{m+1},r_1,\ldots,r_m)$ of
nonnegative numbers with sum $n$.

By Lemma \ref{lemma:count_nice_bijection}, the generating function
of $w_m^n$ is
$$c(x)^{m+1}\big(\frac{g(x)}{x}\big)^m=c(x)^{2m+1}.$$
\end{proof}

\begin{proof}[Proof of \eqref{eq:result_of_straight_menage_numbers}]
We can reduce each permutation in $S_n$ to a straight m\'enage
permutation in $S_i$ ($0\le i\le n$). Thus, we have
$$n!\,=\sum\limits_{i=0}^nw_i^{n-i}V_i$$
where $V_i$ is the $i$th straight m\'enage number and $w_i^{n-i}$ is
the same as in Theorem \ref{gf_of_omega}. Thus,
\begin{align*}
\sum\limits_{n=0}^{\infty}n!\,x^n=\sum\limits_{n=0}^\infty
\sum\limits_{i=0}^nw_i^{n-i}V_ix^n=\sum\limits_{i=0}^\infty
\sum\limits_{n=i}^\infty w_i^{n-i}V_ix^n=\sum\limits_{i=0}^\infty
\sum\limits_{n=0}^\infty w_i^nV_ix^nx^i=\sum\limits_{i=0}^\infty
c(x)^{2i+1}V_ix^i
\end{align*}
where the last equality is from Theorem \ref{gf_of_omega}.
\end{proof}

\section{Structure of ordinary m\'enage permutations}\label{sec:ordinary_menage_permutation} By definition, a permutation $\tau$ is an ordinary m\'enage
permutation if and only if we cannot apply reductions of either type
1 or type 3 on $\tau$. Similarly as in Section
\ref{sec:straight_menage_permutation}, we can reduce each
permutation $\pi$ to an ordinary m\'enage permutation by reductions
of type 1 and type 3. The resulting permutation does not depend on
the order of the reductions.

By the circular diagram representation of permutations, it is not
difficult to see that we can reduce a permutation $\pi$ to
$\pi_\emptyset$ by reductions of type 1 and type 2 if and only if we
can reduce $\pi$ to $\pi_\emptyset$ by reductions of type 1 and type
3. Thus, we have the following lemma:
\begin{lemma}\label{lemma:equivalence}
Suppose $\pi\in S_n$. We can reduce $\pi$ to $\pi_\emptyset$ by
reductions of type 1 and type 3 if and only if there is a
noncrossing partition inducing $\pi$. In particular, there are $C_n$
permutations of this type in $S_n$.
\end{lemma}
In the following parts of Section
\ref{sec:ordinary_menage_permutation}, when we mention reductions,
we mean reductions of \textbf{type 1 or type 3} unless otherwise
specified.
\begin{theorem}\label{gf_of_r}
Suppose $m\ge0$ and $\pi\in S_m$ is an ordinary m\'enage
permutation. Let $r_m^n$ denote the cardinality of the set
\begin{align}\label{set:permutations_which_can_be_reduced_to_a_given_permutation_by_type_1_and_3}
\{\tau\in S_{m+n}|\text{we can reduce }\tau\text{ to }\pi\text{ by
reductions of type 1 and type 3}\}.
\end{align}
Then, the generating function of $r_m^n$ satisfies
\begin{align*}
R_m(x):=r_m^0+r_m^1x+r_m^2x^2+r_m^3x^3+\cdots=\begin{cases}c'(x)c(x)^{2m-2}&\text{if}\quad m>0;\\
c(x)&\text{if}\quad m=0.\end{cases}
\end{align*}
\end{theorem}
\begin{proof}
When $m=0$, $r_m^n=C_n$ by Lemma \ref{lemma:equivalence}. So
$R_0(x)=c(x)$. Now, suppose $m>0$.

Obviously $r_m^0=1$. Suppose $n>0$.

Represent $\pi$ by a circular diagram. The diagram has $m$ gaps: one
gap between each pair of adjacent points. Call the point
corresponding to number $i$ $point$ $i$.

Let $A$ denote the set in
\eqref{set:permutations_which_can_be_reduced_to_a_given_permutation_by_type_1_and_3}.
To obtain a permutation in $A$ we can add points into $\pi$ by the
following steps:
\begin{enumerate}
\item[(a)] Add a permutation induced by a noncrossing partition $\Phi_i$ of
$[d_i]$ into the gap between point $i$ and point $i+1$ (mod $m$),
where $1\le i\le m$ and $d_p\ge0$ ($d_p=0$ means we add nothing into
the gap). Use $Q_i$ to denote the set of points added into the gap
between point $i$ and point $i+1$ (mod $m$).
\item[(b)] Replace point $i$ by a nice bijection $f_i$ from
$[t_i]$ to $\{2,\ldots,t_i+1\}$, where $1\le i\le m$ and $t_i\ge0$
($t_i=0$ means that we do not change the $i$th point). Use $P_i$ to
denote the set of points obtained from this replacement. Thus, $P_i$
contains $t_i+1$ points.
\item[(c)] Specify a point in $P_1\bigcup Q_m$ to correspond to the
number 1 of the new permutation $\pi'$.
\end{enumerate}
Steps (a) and (b) are the same as in the proof of Theorem
\ref{gf_of_omega}, but Step (c) needs some explanation.

In the proof of Theorem \ref{gf_of_omega}, after adding points into
the permutation by (a) and (b), we defined $\pi'$ from the resulting
$horizontal$ diagram in a natural way, that is, the left most point
corresponds to 1, and the following points correspond to 2, 3, 4,
\ldots \,respectively. However, now there is no natural way to
define $\pi'$ from the resulting $circular$ diagram because we have
more than one choice of the point corresponding to number 1 of
$\pi'$.

Note that $\pi'$ can become $\pi$ by a series of reductions of type
1 or type 3. Then, for each $1\le u\le m$, the points in $P_u$ will
become point $u$ of $\pi$ after the reductions. Thus, we can choose
any point of $P_1$ to be the one corresponding to the number 1 of
$\pi'$. Moreover, we can also choose any point of $Q_m$ to be the
one corresponding to the number 1 of $\pi'$. The reason is that if a
permutation in $S_k$ has a cycle $(1,k)$, then a reduction of type 3
will reduce $(1,k)$ to the cycle $(1)$. Thus, we can choose any
point in $P_1\bigcup Q_m$ to be the point corresponding to number 1
of $\pi'$. To see this more clearly, let us look at an example.
Suppose $\pi$ and $f$ are as below:

\centerline{\includegraphics[width=3.2in]{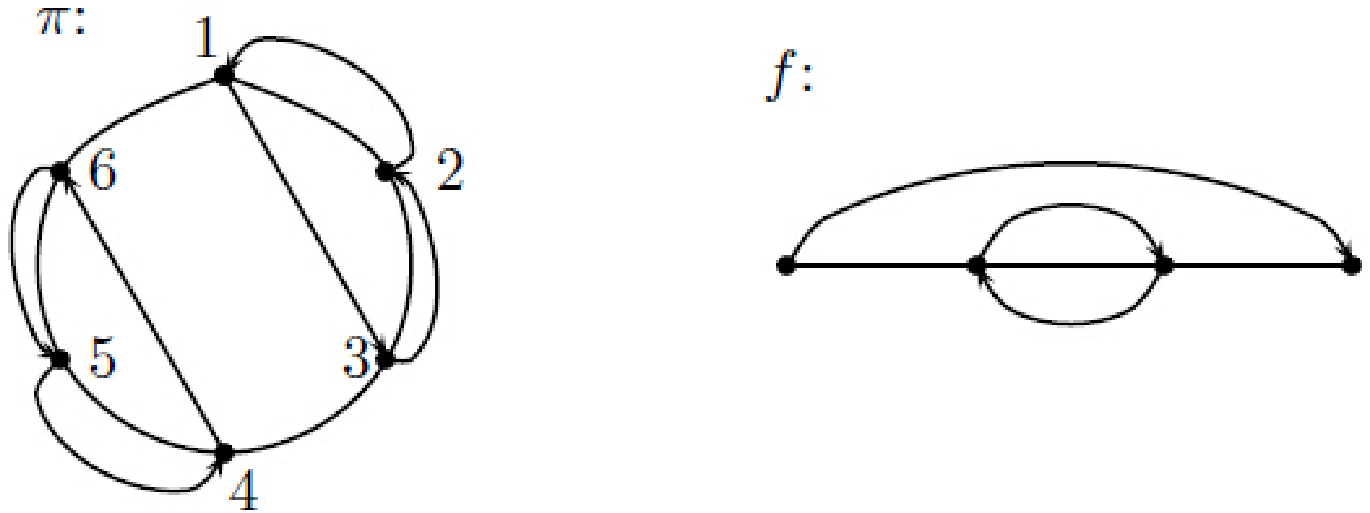}}

Then, we can add the permutation $(1,2)$ between point 2 and point
3, add the permutation $(1)(2,3)$ between point 6 and point 1 and
replace point 2 by $f$. Then, we obtain a new diagram:

\centerline{\includegraphics[width=2in]{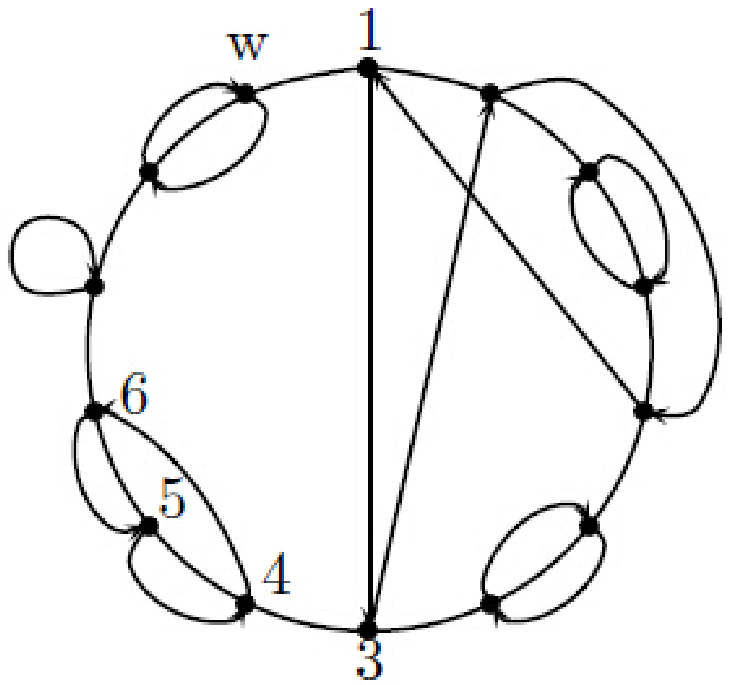}}

Now, we can choose point 1 or any of the three points between point
1 and point 6 to be the point corresponding to number 1 of the new
permutation $\pi'$.

For instance, if we set point 1 to be the point corresponding to
number 1 of $\pi'$, then
$$\pi'=(1,8,2,5)(3,4)(6,7)(9,11,10)(12)(13,14).$$ If we set point w to be the point corresponding to number
1 of $\pi'$, then
$$\pi'=(1,14)(2,9,3,6)(4,5)(7,8)(10,12,11)(13).$$

Now, continue the proof. By a similar argument to the one we used to
prove the statement in the proof of Theorem \ref{gf_of_omega}, we
have that the set of permutations constructed by (a)--(c) equals
$A$.

Now, add points to $\pi$ by (a)--(c). Then, $P_1\bigcup Q_m$
contains, in total, $d_m+(t_1+1)$ points. Thus, we have
$d_m+(t_1+1)$ ways to specify the point corresponding to number 1 in
$\pi'$. The total number added to $\pi$ is
$d_1+\cdots+d_m+t_1+\cdots+t_m$. Therefore, the number of
permutations in $S_{m+n}\cap A$ is
\[
r_m^n=\sum\limits_{d_1,\ldots,d_m\atop t_1,\ldots, t_m}C_{d_1}\cdots
C_{d_m}a_{1+t_1}\cdots a_{1+t_m}(d_m+t_1+1)
\]
where $C_k$ is the $k$th Catalan number, $a_k$ is the same as in
Lemma \ref{lemma:count_nice_bijection}, and the sum runs over all
$2m$-triples $(d_1,\ldots,d_m,t_1,\ldots, t_m)$ of nonnegative
integers such that $\sum_{u=1}^m(t_u+d_u)=n$.

Set $\eta_k=\sum\limits_{r=0}^ka_{r+1}C_{k-r}(k+1)$; from Lemma
\ref{lemma:count_nice_bijection}, we have
$$1+\frac{\eta_1}{2}x+\frac{\eta_2}{3}x^2+\frac{\eta_3}{4}x^3+\cdots=c^2(x).$$
By Lemma \ref{thm:property_of_Catalan_numbers}, the generating
function of $\eta_k$ is
$1+\eta_1x+\eta_2x^2+\eta_3x^3+\cdots=(xc^2(x))'=c'(x)$.

Thus, the generating function of $r_m^n$ is $c(x)^{2m-2}c'(x)$.
\end{proof}

\begin{proof}[Proof of \eqref{eq:result_of_ordinary_menage_numbers}]
We can reduce each permutation in $S_n$ to an ordinary m\'enage
permutation. Thus, we have
$$n!\,=\sum\limits_{i=0}^nr_i^{n-i}U_i$$
where $U_i$ is the $i$th ordinary m\'enage number and $r_i^{n-i}$ is
the same as in Theorem \ref{gf_of_r}. Thus,
\begin{multline*}
\sum\limits_{n=0}^{\infty}n!\,x^n=\sum\limits_{n=0}^\infty
\sum\limits_{i=0}^nr_i^{n-i}U_ix^n=\sum\limits_{i=0}^\infty
\sum\limits_{n=i}^\infty r_i^{n-i}U_ix^n\\=\sum\limits_{i=0}^\infty
\sum\limits_{n=0}^\infty r_i^nU_ix^nx^i
=c(x)+\sum\limits_{i=1}^\infty c'(x)c(x)^{2i-2}U_ix^i
\end{multline*}
where the last equality follows from Theorem \ref{gf_of_r}.
\end{proof}

\section{Counting m\'enage permutations by number of cycles}\label{count_permutations_by_cycles} We prove \eqref{eq:straight_by_cycles} in Section \ref{sec:count_straight_by_cycles} and prove \eqref{eq:ordinary_by_cycles} in
Section \ref{sec:count_ordinary_by_cycles}. Our main method is
coloring. We remark that one can also prove
\eqref{eq:straight_by_cycles} and \eqref{eq:ordinary_by_cycles} by
using the inclusion-exclusion principle in a similar way to
\cite{Gessel}.
\subsection{Coloring and weights}
For a permutation $\pi$, we use $f(\pi)$, $g(\pi)$, $h(\pi)$ and
$r(\pi)$ to denote the number of its cycles, fixed points,
successions and generalized successions, respectively. The following
lemma is well known.

\begin{lemma}\label{lemma:well_known_lemma_for_cycles} For $n\ge0$,
\begin{align*}
\sum\limits_{\pi\in S_n}\alpha^{f(\pi)}=(\alpha)_n.
\end{align*}
\end{lemma}

For a permutation, color some of its fixed points red and color some
of its $generalized$ successions yellow. Then, we obtain a colored
permutation. Here and in the following sections, we use the phrase
$colored$ $permutation$ as follows: if two permutations are the same
as maps but have different colors, then they are different colored
permutations.

Define $\mathbb{S}_n$ to be the set of colored permutations on $n$
objects. Set $A_n$ to be the subset of $\mathbb{S}_n$ consisting of
colored permutations with a colored generalized succession
$\{n,1\}$. Set $B_n=\mathbb{S}_n\backslash A_n$. So we can consider
$S_n$ as a subset of $\mathbb{S}_n$ consisting of permutations with
no color. In particular, $\mathbb{S}_0=S_0$.

Set
$$M_n^\alpha(t,u)=\sum\limits_{\pi\in
S_n}\alpha^{f(\pi)}t^{g(\pi)}u^{h(\pi)}\quad\quad\text{and}\quad\quad
L_n^\alpha(t,u)=\sum\limits_{\pi\in
S_n}\alpha^{f(\pi)}t^{g(\pi)}u^{\tau(\pi)}.$$

For a colored permutation $\epsilon\in\mathbb{S}_n$, define two
weights $W_1$ and $W_2$ of $\epsilon$ as:
\begin{align*}
W_1(\epsilon)&=x^n\cdot\alpha^{f(\epsilon)}\cdot t^{\text{number of
colored
fixed points}}\cdot u^{\text{number of colored successions}},\\
W_2(\epsilon)&=x^n\cdot\alpha^{f(\epsilon)}\cdot t^{\text{number of
colored fixed points}}\cdot u^{\text{number of colored generalized
successions}}.
\end{align*}

\begin{lemma}\label{lemma:sum_of_the_W1_and_W2_weights}
$\sum\limits_{\epsilon\in B_n}W_1(\epsilon)=\sum\limits_{\epsilon\in
B_n}W_2(\epsilon)=M_n^\alpha(1+t,1+u)x^n$,
$\sum\limits_{\epsilon\in\mathbb{S}_n}W_2(\epsilon)=L_n^\alpha(1+t,1+u)x^n$.
\end{lemma}

\begin{proof}
The lemma follows directly from the definitions of $M_n^\alpha$,
$L_n^\alpha$, $B_n$, $W_1$ and $W_2$.
\end{proof}

\subsection{Counting straight m\'enage permutations by number of
cycles}\label{sec:count_straight_by_cycles}

In this subsection, we represent permutations by diagrams of
\textbf{horizontal type}.

Suppose $n\ge0$ and $\pi\in S_n$. Then, $\pi$ has $n+1$ gaps: one
gap before the first point, one gap after the last point and one gap
between each pair of adjacent points. We can add points to $\pi$ by
the following steps.

(a) Add the identity permutation of $S_{(d_p)}$ into the $p$th gap
of $\pi$, where $1\le p\le n+1$ and $d_p\ge0$ ($d_p=0$ means that we
add nothing into the $p$th gap).

(b) Replace the $q$th point of $\pi$ by a nice bijection from
$[r_q]$ to $\{2,\ldots,r_q+1\}$, which sends each $i$ to $i+1$.
Here, $1\le q\le n$ and $r_q\ge0$ ($r_q=0$ means that we do not
change the $q$th point).

(c) Color the fixed points added by (a) red, and color the
successions added by (b) yellow.

Then (a), (b) and (c) give a colored permutation in
$\bigcup_{n=0}^\infty B_n$.
\begin{lemma}\label{lemma:sum_of_W1_weight_of_things_constructed_from_pi} Suppose $\pi\in S_n$. The sum of the $W_1$-weights of all
colored permutations constructed from $\pi$ by (a)--(c) is
\[
x^n\cdot\alpha^{f(\pi)}\cdot\frac{1}{(1-\alpha
tx)^{n+1}}\cdot\frac{1}{(1-ux)^{n}}.
\]
\end{lemma}
\begin{proof}
In Step (a), we added $d_p$ fixed points into the $p$th gap. They
contribute $(xt\alpha)^{d_p}$ to the weight because each of them is
a single cycle and a colored fixed point. Because $d_p$ can be any
nonnegative integer, the total contribution of the fixed points
added into a gap is $\frac{1}{(1-\alpha tx)}$. Thus, the total
contribution of the fixed points added into all the gaps is
$\frac{1}{(1-\alpha tx)^{n+1}}$.

In Step (b), through the replacement on the $q$th point, we added
$r_q$ points and $r_q$ successions to $\pi$ (each of which received
a color in Step (c)). Thus, the contribution of this replacement to
the weight is $(ux)^{d_q}$. Because $d_q$ can be any nonnegative
integer, the total contribution of the nice bijections replacing the
$q$th point is $\frac{1}{(1-ux)}$. Thus, the total contribution of
the nice bijections corresponding to all points is
$\frac{1}{(1-ux)^{n}}$.

Observing that the $W_1$-weight of $\pi$ is
$x^n\cdot\alpha^{f(\pi)}$, we complete the proof.
\end{proof}

Suppose $\epsilon$ is a colored permutation in $\bigcup_{n=0}^\infty
B_n$. If we perform reductions of type 1 on the colored fixed points
and perform reductions of type 2 on the colored successions, then we
obtain a new permutation $\epsilon'$ with no color. This $\epsilon'$
is the only permutation in $\bigcup_{n=0}^\infty S_n$ from which we
can obtain $\epsilon$ by Steps (a)--(c). Therefore, there is a
bijection between $\bigcup_{n=0}^\infty B_n$ and
$$\bigcup_{n=0}^\infty\bigcup_{\pi\in S_n}\{\text{colored permutation
constructed from $\pi$ through Steps (a)--(c)}\}.$$ Because of the
bijection, Lemmas \ref{lemma:sum_of_the_W1_and_W2_weights} and
\ref{lemma:sum_of_W1_weight_of_things_constructed_from_pi} imply
that
\begin{align}\label{kkk}
\sum\limits_{n=0}^\infty
M_n^\alpha(1+t,1+u)x^n=\sum\limits_{n=0}^\infty\sum\limits_{\pi\in
S_n}\dfrac{x^n\alpha^{f(\pi)}}{(1-\alpha tx)^{n+1}(1-ux)^{n}}.
\end{align}
\begin{proof}[Proof of \eqref{eq:straight_by_cycles}]
By Lemma \ref{lemma:well_known_lemma_for_cycles} and \eqref{kkk},
the sum of the $W_1$-weights of colored permutations in
$\bigcup\limits_{n=0}^\infty B_n$ is
\begin{align}\label{lll}
\sum\limits_{n=0}^\infty
M_n^\alpha(1+t,1+u)x^n=\sum\limits_{n=0}^\infty
\frac{x^n(\alpha)_n}{(1-\alpha tx)^{n+1}(1-ux)^{n}}.
\end{align}
Setting $t=u=-1$, we have
\[
\sum\limits_{n=0}^\infty M_n^\alpha(0,0)x^n=\sum\limits_{n=0}^\infty
\frac{x^n(\alpha)_n}{(1+\alpha x)^{n+1}(1+x)^{n}}.
\]

Recall that straight m\'enage permutations are permutations with no
fixed points or successions. Thus, for $n>0$, $M_n^\alpha(0,0)$ is
the sum of the $W_1$-weights of straight m\'enage permutations in
$S_n$, which is $\sum\limits_{j=1}^nC_n^j\alpha^jx^n$. Furthermore,
$M_0^\alpha(0,0)=1$. Thus, we have proved
\eqref{eq:straight_by_cycles}.
\end{proof}

\subsection{Counting ordinary m\'enage permutations by number of
cycles}\label{sec:count_ordinary_by_cycles} In this subsection, we
represent permutations by diagrams of the \textbf{horizontal type}.

Suppose $n\ge0$ and $\pi\in S_n$. Define $\mathbb{S}_m(\pi)$ to be
the a subset of $\mathbb{S}_m$: $\tau\in\mathbb{S}_m$ is in
$\mathbb{S}_m(\pi)$ if and only if when we apply reductions of type
1 on the colored fixed points of $\tau$ and apply reductions of type
3 on the colored generalized successions of $\tau$, we obtain $\pi$.
Define $A_m(\pi)=\mathbb{S}_m(\pi)\cap A_m$ and
$B_m(\pi)=\mathbb{S}_m(\pi)\backslash A_m(\pi)$.

\begin{lemma}\label{lemma:sum_of_W2_weights_in_all_An}
The $W_2$-weights of colored permutations in $\bigcup_{n=0}^\infty
A_n$ are
\begin{align*}
\sum\limits_{n=1}^\infty x^n(\alpha)_n\cdot\frac{1}{(1-\alpha
tx)^n}\cdot\frac{1}{(1-ux)^{n+1}}\cdot ux+\alpha utx+\frac{\alpha
ux}{1-ux}.
\end{align*}
\end{lemma}
\begin{proof}
We first evaluate the sum of $W_2$-weights of colored permutations
in $\bigcup_{m=n}^\infty A_m(\pi)$ and then add them up with respect
to $\pi\in S_n$ and $n\ge0$.

\textbf{Case 1: $n>0$.} In this case, for $\pi\in S_n$, we can
construct a colored permutation in $\bigcup_{m=n}^\infty A_m(\pi)$
by the following steps.

($a^\prime$) Define $\tilde\pi$ to be a permutation in $S_{n+1}$
that sends $\pi^{-1}(1)$ to $n+1$, sends $n+1$ to 1 and sends all
other $j$ to $\pi(j)$. Represent $\tilde\pi$ by a horizontal
diagram.

($b^\prime$) For $1\le p\le n-1$, add the identity permutation of
$S_{(d_p)}$ into the gap of $\tilde\pi$ between number $p$ and
$p+1$, where $d_p\ge0$ ($d_p=0$ means we add nothing into the gap).

($c^\prime$) Replace the $q$th point of $\tilde\pi$ by a nice
bijection from $[r_q]$ to $\{2,\ldots,r_q+1\}$, which maps each $i$
to $i+1$. Here, $1\le q\le n$ and $r_q\ge0$ ($r_q=0$ means that we
do not change the $q$th point).

($d^\prime$) Color the generalized succession consisting of 1 and
the largest number yellow. Color the fixed points and generalized
successions added by ($b^\prime$)--($c^\prime$) red and yellow,
respectively.

Suppose $\epsilon\in\bigcup_{m=n}^\infty A_m(\pi)$. If we perform
reductions of type 1 on its colored fixed points and perform
reductions of type 3 on its colored generalized successions, we
obtain $\pi$. Furthermore, $\pi$ is the only permutation in
$\bigcup_{k=0}^\infty S_k$ from which we can obtain $\epsilon$
through ($a^\prime$)--($d^\prime$). Therefore, there is a bijection
between $\bigcup_{m=n}^\infty A_m(\pi)$ and
$$\{\text{colored
permutations constructed from $\pi$ through
($a^\prime$)--($d^\prime$)}\}.$$

Now, we can claim that the sum of $W_2$-weight of the colored
permutations in $\bigcup_{m=n}^\infty A_m(\pi)$ is
\begin{align}\label{W2_weight_of...}
x^n\alpha^{f(\pi)}\cdot\frac{1}{(1-\alpha
tx)^n}\cdot\frac{1}{(1-ux)^{n+1}}\cdot ux.
\end{align}
In \eqref{W2_weight_of...}, $x^n\alpha^{f(\pi)}$ is the $W_2$-weight
of $\pi$. The term $\frac{1}{(1-\alpha tx)^n}$ corresponds to the
fixed points added to the permutation in ($b^\prime$). The term
$\frac{1}{(1-ux)^{n+1}}$ corresponds to the successions added to the
permutation in ($c^\prime$). The term $ux$ corresponds to the
generalized succession $\{n+1,1\}$ added to the permutation in
($a^\prime$).

\textbf{Case 2: $n=0$ and $m=0$.} In this case, $A_0(\pi_\emptyset)$
is empty.

\textbf{Case 3: $n=0$ and $m\ge2$.} In this case,
$A_m(\pi_\emptyset)$ contains one element: the cyclic permutation
$\pi_C$, which maps each $i$ to $i+1$ (mod $m$). Each generalized
succession of $\pi_C$ has a color, so the sum of the $W_2$-weight of
the colored permutations in $A_m(\pi_\emptyset)$ is $\alpha(ux)^m$.

\textbf{Case 4: $n=0$ and $m=1$.} In this case, $A_1(\pi_\emptyset)$
contains one map: $id_1\in S_1$. However, $id_1$ can have two types
of color, namely, yellow and red+yellow, because $id_1$ has one
fixed point and one generalized succession. Thus,
$A_1(\pi_\emptyset)$ contains two colored permutations, and the sum
of their $W_2$-weights is $\alpha ux+\alpha tux$.

We remark that the identity permutation $id_1$ actually corresponds
to four colored permutations. In addition to the two in
$A_1(\pi_\emptyset)$, the other two are $id_1$, with a red color for
its fixed point, and $id_1$, with no color. We have considered these
colored permutations in $B_1(\pi_\emptyset)$ and $B_1(id_1)$,
respectively.

Because $\bigcup_{n=0}^\infty
A_n=\bigcup_{n=0}^\infty\bigcup_{\pi\in S_n}\bigcup_{m=n}^\infty
A_m(\pi)$, the sum of the $W_2$-weights of the colored permutations
in $\bigcup_{n=0}^\infty A_n$ equals the sum of the weights found in
Case 1--4. Thus, the sum is
\begin{multline*}
\Bigg(\sum\limits_{n=1}^\infty\sum\limits_{\pi\in
S_n}x^n\alpha^{f(\pi)}\cdot\frac{1}{(1-\alpha
tx)^n}\cdot\frac{1}{(1-ux)^{n+1}}\cdot
ux\Bigg)+\big(\sum\limits_{m=2}^\infty\alpha(ux)^m\big)+\big(\alpha
ux+\alpha
tux\big)\\
=\sum\limits_{n=1}^\infty x^n(\alpha)_n\cdot\frac{1}{(1-\alpha
tx)^n}\cdot\frac{1}{(1-ux)^{n+1}}\cdot ux+\alpha utx+\frac{\alpha
ux}{1-ux}
\end{multline*}
where we used Lemma \ref{lemma:well_known_lemma_for_cycles}.
\end{proof}

\begin{proof}[Proof of \eqref{eq:ordinary_by_cycles}]
By Lemma \ref{lemma:sum_of_the_W1_and_W2_weights},
$\sum\limits_{n=0}^\infty L_n^\alpha(1+t,1+u)x^n$ is the sum of the
$W_2$-weights of the colored permutations in
$\bigcup_{n=0}^\infty\mathbb{S}_n$. By Lemma
\ref{lemma:sum_of_the_W1_and_W2_weights} and \eqref{lll}, the sum of
the $W_2$-weights of all colored permutations in
$\bigcup_{n=0}^\infty B_n$ is
$$\sum\limits_{n=0}^\infty \frac{x^n(\alpha)_n}{(1-\alpha
tx)^{n+1}(1-ux)^{n}}.$$ Because
$\bigcup_{n=0}^\infty\mathbb{S}_n=(\bigcup_{n=0}^\infty
B_n)\bigcup(\bigcup_{n=0}^\infty A_n)$, Lemma
\ref{lemma:sum_of_W2_weights_in_all_An} implies

\begin{align}\label{eq:equation}
&\sum\limits_{n=0}^\infty L_n^\alpha(1+t,1+u)x^n\nonumber\\
=&\Bigg(\sum\limits_{n=0}^\infty \frac{x^n(\alpha)_n}{(1-\alpha
tx)^{n+1}(1-ux)^{n}}\Bigg)+\Bigg(\sum\limits_{n=1}^\infty
x^n(\alpha)_n\cdot\frac{1}{(1-\alpha
tx)^n}\cdot\frac{1}{(1-ux)^{n+1}}\cdot ux+\alpha utx+\frac{\alpha
ux}{1-ux}\Bigg)\nonumber\\
=&\sum\limits_{n=0}^\infty \bigg[\frac{x^n(\alpha)_n}{(1-\alpha
tx)^{n+1}(1-ux)^{n+1}}(1-\alpha tux^2)\bigg]+\alpha
utx+\frac{(\alpha-1) ux}{1-ux}.
\end{align}
Recall that ordinary m\'enage permutations are permutations with no
fixed points and no generalized successions. By definition,
$L_0^\alpha(0,0)x^0=1$. When $n\ge1$, $L_n^\alpha(0,0)x^n$ is the
sum of the $W_2$-weights of all ordinary m\'enage permutations in
$S_n$, which equals $\sum\limits_{j=1}^nD_n^j\alpha^jx^n$. Thus,
$\sum\limits_{n=0}^\infty L_n^\alpha(0,0)x^n$ equals the left side
of \eqref{eq:ordinary_by_cycles}. When we set $t=u=-1$, the left
side of \eqref{eq:equation} equals the left side of
\eqref{eq:ordinary_by_cycles}, and the right side of
\eqref{eq:equation} equals the right side of
\eqref{eq:ordinary_by_cycles}. We have proved
\eqref{eq:ordinary_by_cycles}.
\end{proof}

\section{An analytical proof of \eqref{eq:result_of_straight_menage_numbers} and
\eqref{eq:result_of_ordinary_menage_numbers}}\label{appendix}

Now, we derive \eqref{eq:result_of_straight_menage_numbers} and
\eqref{eq:result_of_ordinary_menage_numbers} from
\eqref{eq:known_formulas_for_U_n} and
\eqref{eq:known_formulas_for_V_n}. By
\eqref{eq:known_formulas_for_V_n},
\begin{align}\label{qqq}
\sum\limits_{n=0}^\infty
V_nx^n&=\sum\limits_{n=0}^\infty\sum\limits_{k=0}^n(-1)^k{2n-k\choose
k}(n-k)!\,x^n=\sum\limits_{k=0}^\infty\sum\limits_{n=k}^\infty(-1)^k{2n-k\choose
k}(n-k)!\,x^n\nonumber\\&=\sum\limits_{k=0}^\infty\sum\limits_{n=0}^\infty(-1)^k{2n+k\choose
k}n!\,x^{n+k}=\sum\limits_{n=0}^\infty
n!\,x^n\bigg[\sum\limits_{k=0}^\infty(-1)^k{2n+k\choose
k}x^k\bigg]\nonumber\\&=\sum\limits_{n=0}^\infty
n!\,\dfrac{x^n}{(1+x)^{2n+1}}.
\end{align}
Letting $x=zc^2(z)$, from Lemma
\ref{thm:property_of_Catalan_numbers} we have $1+x=c(z)$,
$\dfrac{x}{(1+x)^2}=z$ and
\begin{align*}
\sum\limits_{n=0}^\infty V_nz^nc^{2n}(z)=\sum\limits_{n=0}^\infty
V_nx^n=\sum\limits_{n=0}^\infty
n!\,\dfrac{x^n}{(1+x)^{2n+1}}=\sum\limits_{n=0}^\infty
n!\,\dfrac{z^n}{c(z)}.
\end{align*}
which implies \eqref{eq:result_of_straight_menage_numbers}. From
\eqref{eq:known_formulas_for_U_n},
\begin{align*}
U_n=\begin{cases}1&\text{if }n=0;\\0&\text{if
}n=1;\\\sum\limits_{k=0}^n(-1)^k{2n-k\choose
k}(n-k)!+\sum\limits_{k=1}^n(-1)^k{2n-k-1\choose k-1}(n-k)!&\text{if
}n>1.\end{cases}
\end{align*}
Thus
\begin{align*}
\sum\limits_{n=0}^\infty U_nx^n
&=x+\sum\limits_{n=0}^\infty\sum\limits_{k=0}^n(-1)^k{2n-k\choose
k}(n-k)!\,x^n+\sum\limits_{n=1}^\infty\sum\limits_{k=1}^n(-1)^k{2n-k-1\choose
k-1}(n-k)!\,x^n\\
&=x+\sum\limits_{n=0}^\infty\sum\limits_{k=0}^n(-1)^k{2n-k\choose
k}(n-k)!\,x^n+\sum\limits_{n=0}^\infty\sum\limits_{k=0}^n(-1)^{k+1}{2n-k\choose
k}(n-k)!\,x^{n+1}\\
&=x+(1-x)\sum\limits_{n=0}^\infty\sum\limits_{k=0}^n(-1)^k{2n-k\choose
k}(n-k)!\,x^n\nonumber.
\end{align*}
Then, by \eqref{qqq},
\begin{align*}
\sum\limits_{n=0}^\infty U_nx^n=x+(1-x)\sum\limits_{n=0}^\infty
n!\,\dfrac{x^n}{(1+x)^{2n+1}}.
\end{align*}
Noticing $U_0=1$ we have $$1-x+\sum\limits_{n=1}^\infty
U_nx^n=(1-x)\sum\limits_{n=0}^\infty n!\,\dfrac{x^n}{(1+x)^{2n+1}}$$
and
\begin{align}\label{eq:for_analytical_proof_of_ordinary}
1+x+\dfrac{1+x}{1-x}\sum\limits_{n=1}^\infty
U_nx^n=\sum\limits_{n=0}^\infty n!\,\dfrac{x^n}{(1+x)^{2n}}.
\end{align}
Letting $x=zc^2(z)$, from Lemma
\ref{thm:property_of_Catalan_numbers} and
\eqref{eq:for_analytical_proof_of_ordinary}, we have $1+x=c(z)$,
$\dfrac{x}{(1+x)^2}=z$ and
\begin{align*}
\sum\limits_{n=0}^\infty n!\,z^n=c(z)+\dfrac{c(z)}{1-zc^2(z)}\sum\limits_{n=1}^\infty
U_nz^n(c(z))^{2n}=c(z)+\dfrac{(c(z))^3}{1-zc^2(z)}\sum\limits_{n=1}^\infty
U_nz^n(c(z))^{2n-2}.
\end{align*}
Then, \eqref{eq:result_of_ordinary_menage_numbers} follows from the
above equation and Lemma \ref{thm:property_of_Catalan_numbers}.

\section{Acknowledgements}
It is a pleasure to thank Ira Gessel for introducing me to this
topic and for some helpful conversations. I also thank Olivier
Bernardi for some helpful suggestions.

\Addresses


\begin{thebibliography}{99}

\bibitem{Bogart} K. Bogart and G. Doyle, Non-sexist solution of the
m\'enage problem, {\it Amer. Math. Monthly} {\bf 93} (1986),
514--518.

\bibitem{CW} E. Canfield and N. Wormald, M\'enage numbers,
bijections and precursiveness, {\it Discrete Math.} {\bf 63} (1987),
117--129.

\bibitem{Gessel} I. Gessel, Generalized rook polynomials and orthogonal
polynomials, in {\it IMA Volumes in Mathematics and its
Applications}, Vol.\ 18, Springer-Verlag, 1989, pp.~159--176.

\bibitem{Kaplansky1943} I. Kaplansky, Solution of the probl¨¨me des m\'enages, {\it Bull. Amer. Math. Soc.} {\bf 49} (1943), 784--785.

\bibitem{Kaplansky} I. Kaplansky and J. Riordan, The probl\`eme des
m\'enages, {\it Scripta Math.} {\bf 12} (1946), 113--124.

\bibitem{Lucas} E. Lucas, {\it Th\'eorie des Nombres}, Gauthier-Villars, 1891.

\bibitem{MW} L. Moser and M. Wyman, On the probl\`eme des m\'enages, {\it Canad. J. Math.} {\bf 10} (1958), 468--480.

\bibitem{Riordan} J. Riordan, {\it An Introduction to Combinatorial
Analysis}, Wiley, 1958.

\bibitem{Stanley} R. Stanley, {\it Enumerative Combinatorics}, Vol.\ 2, Cambridge University Press, 1999.

\bibitem{Touchard} J. Touchard , Sur un probl\`eme de permutations, {\it C. R. Acad. Sci. Paris} {\bf 198} (1934), 631--633.

\end{thebibliography}
\end{document}